\documentclass[letterpaper, 10 pt, conference]{ieeeconf}

\usepackage{flushend}

\usepackage{amssymb}
\usepackage[cmex10]{amsmath}
\usepackage{algorithm}
\usepackage{algpseudocode}
\usepackage{graphicx}
\usepackage{color}
\usepackage{subcaption}

\IEEEoverridecommandlockouts
\newcommand{\mR}{\ensuremath{\mathbb{R}}}
\newcommand{\Lzi}[1]{\ensuremath{L_{[1,#1)}}}
\newcommand{\Lij}[2]{\ensuremath{L_{(#1,#2)}}}
\newcommand{\Ljn}[1]{\ensuremath{L_{(#1,n]}}}

\newcommand{\Merr}[1]{\ensuremath{R(#1)}}

\newcommand{\xt}[1]{\ensuremath{\overline{x}_{#1}}}
\newcommand{\xtv}{\ensuremath{\overline{\textbf{x}}}}
\newcommand{\Lff}{\ensuremath{L_{\textit{ff}}}}

\newcommand{\pos}[2]{\ensuremath{p_{#1#2}}}

\newcommand{\Mmax}{\ensuremath{R_{\text{max}}}}
\newcommand{\Mopt}{\ensuremath{R^*}}
\newcommand{\Ne}[1]{\ensuremath{N(#1)}}
\newcommand{\expec}[1]{\textbf{E}\left[#1\right]}
\newcommand{\tr}[1]{\textbf{tr}\left(#1\right)}

\newcommand{\eqdef}{\ensuremath{\triangleq}}
\newtheorem{theorem}{Theorem}

\newtheorem{problem}{Problem Statement}

\newcommand{\Deg}[1]{\ensuremath{\Delta_{#1}}}

\newcommand{\Gb}{\ensuremath{\overline{G}}}
\newcommand{\Eb}{\ensuremath{\overline{E}}}
\newcommand{\Vb}{\ensuremath{\overline{V}}}

\newcommand{\Gbb}{\ensuremath{\overline{G}_i}}
\newcommand{\Ebb}{\ensuremath{\overline{E}_i}}
\newcommand{\Vbb}{\ensuremath{\overline{V}_i}}
\newcommand{\Wbb}{\ensuremath{\overline{W}_i}}

\newcommand{\Mod}[1]{\ (\text{mod}\ #1)}

\title{\LARGE \bf
Efficient, Optimal $k$-Leader Selection for Coherent, \\
One-Dimensional Formations
}

\author{Stacy Patterson$^{1}$,  Neil McGlohon$^{1}$, and Kirill Dyagilev$^{2}$
\thanks{$^{1}$S. Patterson and N. McGlohon are with the Department of Computer Science, Rensselaer Polytechnic Institute,
Troy, NY 12180, USA
        {\tt\small sep@cs.rpi.edu}, {\tt\small mcglon@rpi.edu}}%
\thanks{$^{2}$K. Dyagilev is with the Department of Computer Science, Johns Hopkins University,
        Baltimore, MD 21218, USA
        {\tt\small kirilld@cs.jhu.edu}}%
}

\begin{document}

\maketitle
\thispagestyle{empty}
\pagestyle{empty}

\begin{abstract}

We study the problem of optimal leader selection in consensus networks with noisy relative information. The objective is to identify the set of $k$ leaders that minimizes the formation's deviation from the desired trajectory established by the leaders.  An optimal leader set can be found by an exhaustive search over all possible leader sets; however, this approach is not scalable to large networks. In recent years, several works have proposed approximation algorithms to the $k$-leader selection problem, yet the question of whether there exists an efficient, non-combinatorial method to identify the optimal leader set remains open. This work takes a first step towards answering this question.  We show that, in one-dimensional weighted graphs, namely path graphs and ring graphs, the $k$-leader selection problem can be solved in polynomial time (in both $k$ and the network size $n$). We give an $O(n^3)$ solution for optimal $k$-leader selection in path graphs and an $O(kn^3)$ solution for optimal $k$-leader selection in ring graphs.
\end{abstract}

\section{INTRODUCTION}
We explore the problem of leader selection in leader-follower consensus systems.  Such systems arise in the context of
vehicle formation control~\cite{RBM05}, distributed clock synchronization~\cite{EKPS04}, and distributed localization in sensor networks~\cite{BH09}, among others.
In these systems, several agents act as \emph{leaders} whose states serve as the reference trajectory for the entire system.  The leaders may be controlled autonomously or by a system owner.
The remaining agents are \emph{followers}.  Each follower updates its state based on noisy measurements of the states of its neighbors.
The objective of the leader-follower system is for the entire formation to maintain a desired global state.

Since the follower agents' measurements are corrupted by stochastic noise, the agents cannot maintain the formation exactly.  However,
the variance of the deviation of the agents' states from the desired states is bounded~\cite{BH06}.  This variance is related to the \emph{coherence} of the formation,
and is quantified by an $H_2$ norm of the leader-follower system~\cite{PB10,B12}.  It has been shown that the coherence of such leader-follower consensus systems
 depends  on which agents act as leaders~\cite{BH06,PB10}, and so, by judiciously choosing the leader set, one can minimize the total variance of the formation.
  The \emph{$k$-leader selection problem} is precisely to select a set of at most $k$ leaders for which the total variance of the deviation of the follower nodes is minimized.

The problem of finding a leader set that minimizes the total variance of the deviation has received attention in recent years.
The optimal leader set can be found by an exhaustive search over all subsets of agents of size at most $k$, but this solution is not tractable in large networks.
In~\cite{LFJ14}, the authors present a leader selection algorithm based on a convex-relaxation of the problem.   They then give an efficient algorithm for this relaxed problem.
While this algorithm offers no guarantees on the optimality of the chosen leader set, its good performance was demonstrated on example networks.
 In~\cite{CBP14}, the authors show that the total variance of the deviation from the desired trajectory is a super-modular set function~\cite{N78}.
 This super-modularity property implies that one can use a greedy, polynomial-time algorithm to find a leader for which the total variance is within a provable bound of optimal.
Another recent work~\cite{FL13} has shown a connection between the optimal leader set and the information centrality measures.  They have used this connection to give efficient algorithms for finding
the optimal single leader and optimal pair of leaders in several network topologies.
Other works have explored optimal leader selection in leader-follower systems without stochastic disturbances~\cite{CABP14} and in systems where both
the leaders and followers are subject to stochastic disturbances and the leaders also have access to relative state information~\cite{LFJ14}.
Despite this recent interest, the question of whether it is possible to find an optimal leader set using a non-combinatorial approach remains open.

In this work, we take a first step towards answering this question.  We show that, in one-dimensional weighted graphs, namely path graphs and ring graphs, the $k$-leader selection problem
can be solved in time that is polynomial in both $k$ and the network size $n$.
Our approach is to first transform the $k$-leader selection problem into the problem of finding a minimum-weight path
in a weighted directed graph, a problem that can be solved polynomial time.
We then use a slightly modified version of the well-known Bellman-Ford algorithm~\cite{CLRS10}
to find this minimum-weight path and thus the optimal leader set.  For path graphs, our algorithm finds the optimal leader set of size at most $k$ in
 $O(n^3)$ time.  In ring graphs, our algorithm finds the optimal leader set of size at most $k$ in $O(kn^3)$ time.

Our approach was inspired by recent work on facility location on the real line~\cite{CW14}.
In this work, the authors solve the facility location problem by reducing it to a minimum-weight path problem over a directed acyclic graph.
We note that our graph construction and our path-finding algorithm are both different from those used for the facility location problem.

The remainder of this paper is organized as follows.  In Section~\ref{problem.sec}, we present the system model and formalize the $k$-leader selection problem.
In Section~\ref{algorithm.sec}, we present our  algorithms for finding the optimal leader set in path and ring graphs.
Section~\ref{examples.sec} gives computational examples comparing the optimal leader set to the leader set selected by the greedy algorithm.
Finally, we conclude in Section~\ref{conclusion.sec}.

\section{Problem Formulation} \label{problem.sec}

In this section, we describe the dynamics of leader-follower consensus system and formally define the $k$-leader selection problem.

\subsection{System Model}
We consider a network of agents, modeled by an undirected, connected graph $G=(V,E)$, where the node set $V = \{1, 2, \ldots, n\}$ represents the
agents and the edge set $E$ captures the communication structure of a network.
We use the terms node and agent interchangeably.
If $(i,j) \in E$, then nodes $i$ and $j$ can exchange information.We denote the neighbor set of a node $i$ by $\Ne{i}$, i.e., $\Ne{i} \eqdef \{ j \in V~|~(i,j) \in E\}$.
In this work, we restrict our study to one-dimensional graphs,
namely path graphs and ring graphs.
For path graphs, we assume the nodes IDs are assigned in order along the path.
For ring graphs, we assume that the node IDs are assigned in ascending order around the ring in a clockwise fashion.

Each agent $i$ has a state $x_i(t) \in \mR$ which represents, for example, the position of
agent $i$ at time $t$.  The objective is for each pair of neighbor agents $i$ and $j$ to maintain a pre-specified difference $\pos{i}{j}$ between their states,
\begin{equation} \label{diff.eq}
x_i(t) - x_j(t) = \pos{i}{j}~~~~\text{for all}~(i,j) \in E.
\end{equation}
If $x_i(t)$ represents an agent's position, $\pos{i}{j}$ is the desired distance between agents $i$ and $j$.
We assume that each agent $i$ knows $\pos{i}{j}$ for all $j \in \Ne{i}$.

A subset $S \subset V$ of agents act as leaders.  The  states of these agents
serve as reference states for the network.   The state of each leader $s \in S$ remains fixed at its reference value
$\xt{s}$.
 Let $\xtv$ denote the vector of states that satisfy  (\ref{diff.eq}) when the leader states are fixed at their reference values.
We assume that at least one such $\xtv$ exists.

The remaining agents $v \in V \setminus S$ are followers.
A follower  updates its state based on noisy measurements of the differences between its state and the states
of its neighbors. The dynamics of each follower agent is given by
\[
\dot{x}_i(t) = -\sum_{j \in \Ne{i}} W_{ij}  \left(x_i(t) - x_j(t) - \pos{i}{j}+ \epsilon_{ij}(t)\right),
\]
where $W_{ij}$ is the weight for link $(i,j)$ and
 $\epsilon_{ij}(t)$, $(i,j) \in E$, are  zero-mean white noise processes with autocorrelation functions
$\expec{\epsilon_{ij}(t)\epsilon_{ij}(t + \tau)} = \nu_{ij} \delta(\tau)$.  Here $\delta(\cdot)$ denotes the unit impulse function.
 Each $\epsilon_{ij}$ is independent, and
$\epsilon_{ij}$  and $\epsilon_{ji}$ are identically distributed.
As in~\cite{CBP14}, we select the edges weights as  $W_{ij} = \frac{\nu_{ij}}{\Deg{i}}$, where $\Deg{i} = \sum_{j \in \Ne{i}}(1/\nu_{ij})$.
 This edge weight policy  corresponds to the best linear unbiased estimator of the leader agents' states when $x_j(t) = \xt{j}$ for all $j \in \Ne{i}$~\cite{BH06}.

  \subsection{Performance Measure}
  Without the noise processes $\epsilon_{ij}$, the agents' states would converge to $\xtv$.
With these noise processes, the agents' states deviate from $\xtv$; however, the variances of these deviations are bounded in the mean-square sense.
For a follower agent $i$, let $r_i$ be the steady-state variance of the deviation from $\xt{i}$,
\[
r_i \eqdef \lim_{t \rightarrow \infty} \expec{(x_i(t) - \xt{i})^2}.
\]
The steady-state variances $r_i$,  $i \in V \setminus S$ can be obtained from the weighted Laplacian matrix of the graph $G$,
whose components are defined as
\[
L_{ij} = \left\{ \begin{array}{ll}
- \frac{1}{\nu_{ij}} & ~\text{if}~(i,j) \in E \\
\Deg{i} &~\text{if}~i=j \\
0 &~\text{otherwise.}
\end{array} \right.
\]
For a leader set $S$, let the follower-follower sub-matrix $\Lff$ be the matrix $L$ with the rows and columns corresponding
nodes in $S$ removed.
It has been shown that~\cite{BH06},
\begin{equation} \label{ri.eq}
r_i = \frac{1}{2} (\Lff^{-1})_{ii}.
\end{equation}
We note that since $G$ is connected, $\Lff$ is positive definite~\cite{PB10}, and therefore, (\ref{ri.eq}) is well-defined.

We measure the performance of the leader-follower system for a given leader set $S$ by the total steady-state variance of the deviation from $\xtv$,
\begin{equation} \label{Merr.eq}
\Merr{S} \eqdef \sum_{i \in V \setminus S} r_i = \frac{1}{2} \tr{\Lff^{-1}}
\end{equation}
A formation with a small $\Merr{S}$ exhibits good coherence, i.e., the formation closely resembles a rigid formation.

\subsection{The $k$-Leader Selection Problem}
A natural question that arises is how to identify a leader set $S$ of a certain size that minimizes the total steady-state variance in (\ref{Merr.eq}).
This question is formalized as the \emph{$k$-leader selection problem}.
\begin{problem}
The \emph{$k$-leader selection problem} is
\begin{equation} \label{problem.eq}
\begin{array}{ll}
\text{minimize}~& \Merr{S} \\
\text{subject to}~& |S|  \leq k
\end{array}
\end{equation}
\end{problem}
\vspace{.1cm}

A na\"{\i}ve solution to this problem is to construct all subsets of $V$ of size at most $k$, evaluate $R(S)$ for each leader set,
and choose the set $S^*$ for which $R(S)$ in minimized.  The computational complexity of this solution is combinatorial since the number of leader sets that would need to be evaluated is
$\sum_{i=1}^k {n \choose i}$.  In the next section, we show that for one-dimensional graphs, the optimal leader set can be found with time complexity that is polynomial in both $n$ and $k$.

\begin{figure}
\begin{center}
\includegraphics[scale=.42]{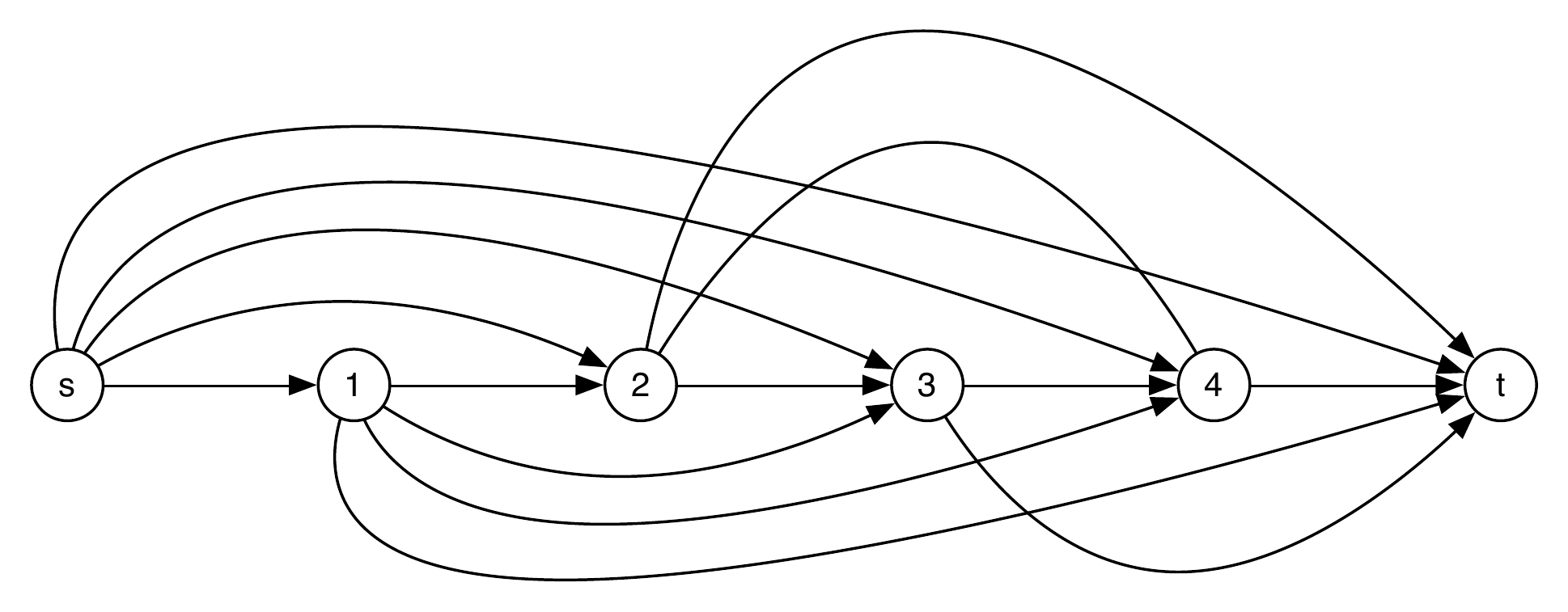}
\end{center}
\caption{Digraph generated from four node path graph.}  \label{dag.fig}
\end{figure}

\section{Leader Selection Algorithm} \label{algorithm.sec}

In this section, we present our leader selection algorithms for path and ring graphs.
Our approach for both graph types is to first reduce the leader selection problem to the problem of finding a minimum-weight path in a
digraph.  We then solve this minimum-weight path problem using a modified version of the
Bellman-Ford algorithm~\cite{CLRS10}.

\begin{algorithm}
\caption{Algorithm for finding the optimal solution to the $k$-leader selection problem in a path graph.} \label{optpath.alg}
\begin{algorithmic}
\State \textbf{Input:} $G = (V,E)$, edge weights $\frac{1}{\nu_{ij}}$, $k$
\State \textbf{Output:} Set of leader nodes $S$, error $R(S)$
~\\
\State $L \gets$ Laplacian of $G$
\State $\Gb \gets $ digraph constructed from $L$
\State $(P, w) \gets \textsc{ModifiedBellmanFord}(\Gb, s, t, k+1)$
~\\
\State \textbf{Construct leader set from min. weight path:}
\State $S \gets \emptyset$
\For{$v \in P$}
\If{$v \neq s$ \text{and} $v \neq t$}
\State $S \gets S \cup \{v\}$
\EndIf
\EndFor
\State $R(S) \gets w$
\State \textbf{return} $(S, R(S))$
\end{algorithmic}
\end{algorithm}

\subsection{Leader Selection for a Path Graph}

We first note that, in a path graph with $k$ leader nodes, the matrix $\Lff$ is block diagonal with at most $k+1$ blocks.  Further,
each block is tridiagonal.
For example, consider a path graph with leader set $S = \{ \ell_1, \ell_2, \ldots, \ell_k\}$, where $\ell_i < \ell_{i+1}$ for $i=1, \ldots, (k-1)$.
The matrix $\Lff$ is formed from  the weighted Laplacian $L$ by removing the rows and columns corresponding to the nodes in $S$.
$\Lff$ can be written as,
\[
\Lff = \left[  \begin{array}{ccccc}
\Lzi{\ell_1} & \textbf{0} & \cdots & \cdots & \textbf{0} \\
\textbf{0} & \Lij{\ell_1}{\ell_2} & \textbf{0} & \cdots & \vdots \\
\vdots &  \ddots & \ddots & \ddots &   \vdots \\
\textbf{0} & \ddots & \ddots & \Lij{\ell_{k-1}}{\ell_k} & \textbf{0} \\
\textbf{0} & \cdots & \cdots & \textbf{0} & \Ljn{\ell_k}
\end{array} \right] ,
\]
where the matrices $\Lzi{\ell_1}$, $\Lij{\ell_i}{\ell_{i+1}}$, $i=1, \ldots ,(k-1)$, and $\Ljn{\ell_2}$ are defined as follows.
$\Lzi{\ell_1}$ is the sub-matrix of $L$ consisting of the rows and columns $1$ through $\ell_1-1$ of $L$.
This matrix models the interactions between the follower nodes, $i =1, \ldots, (\ell_1-1)$.
Note that node $\ell_1$ is the only leader that affects the states of these nodes.
$\Lij{\ell_i}{\ell_{i+1}}$  is the sub-matrix of $L$ consisting of the rows and columns indexed from $\ell_i+1$ to $\ell_{i+1}-1$, inclusive.
This matrix models the interactions of the nodes between leader node $\ell_i$ and leader node $\ell_{i+1}$;
these follower nodes are influenced by both leaders.
Finally, $\Ljn{\ell_k}$ consists of the rows and columns $i=(\ell_k+1),\ldots, n$ of $L$, inclusive, and models the interactions of the follower nodes $i=(\ell_k+1), \ldots, n$.
Leader node $\ell_k$ is the only leader that influences the states of these nodes.
If there are leaders $u, v \in S$ such that $v = u+1$, then the corresponding sub-matrix of $L$ will be of size 0.

The corresponding error $R(S)$ is  given by
\begin{align} \label{R1d.eq}
R(S) &= \frac{1}{2} \tr{{\Lzi{\ell_1}}^{-1}} + \frac{1}{2} \sum_{i=1}^{k-1} \tr{ {\Lij{\ell_i}{\ell_{i+1}}}^{-1}}  \\
&~~~~~+ \frac{1}{2} \tr{{\Ljn{\ell_k}}^{-1}} \nonumber
\end{align}
Note that, if any sub-matrix of $L$ is of size 0, then the trace of the inverse of this sub-matrix is 0.  Similarly, $\tr{{\Lzi{1}}^{-1}} = 0$ and $\tr{{\Ljn{n}}^{-1}} = 0$.
By decomposing $R(S)$ in this manner, we observe that, for a given leader set $S$, the total steady-variance can be computed by finding the trace of the inverse of the
$k+1$ blocks.

We now construct a weighted digraph $\Gb = (\Vb, \Eb, W)$ based on $L$ as follows.
The set of nodes in the graph consists of the nodes in $V$ and an additional source node $s$ and target node $t$, i.e.,
$\Vb = V \cup \{s,t\}$, as shown in Figure~\ref{dag.fig}.
The edge set $\Eb$  contains edges from $s$ to every node $v \in V$.  The weight of edge $(s,v)$ is
$\tr{{\Lzi{v}}^{-1}}$.  This edge weight is the total steady-state variance  for nodes $1, \ldots, (v-1)$ when
node $v$ is a leader node and there are no other leader nodes $u$ with $u < v$.  $\Eb$ also contains edges from from each node $u \in V$ to
each node $v \in V$ with $u < v$.  The weight of edge $(u,v)$ is $\tr{{\Lij{u}{v}}	^{-1}}$.  This weight is the total variance of the nodes
$i=(u+1), \ldots,(v-1)$ when nodes $u$ and $v$ are leaders and there are no other leader nodes $w$ with $u < w < v$.
Finally,  $\Eb$ contains edges from every node $v \in V$ to node $t$.  The weight of edge $(v,t)$ is $\tr{{\Ljn{v}}^{-1}}$.
This weight is the total variance of the nodes $i={v+1}, \ldots, n$ when $v$ is a leader and there are no other leader nodes $u$ with
$u > v$.  Note that the weights of edges $(s,1)$, $(n,t)$, and $(u,u+1),~u=1 \ldots (n-1)$ are 0.

We observe that the weight of a path of length $k+1$  from $s$ to $t$ in $\Gb$ is equivalent to $R(S)$ in (\ref{R1d.eq}),
where each vertex on the path, excepting $s$ and $t$, is an element of $S$.
Thus, to find a leader set $S$  with $|S| \leq k$ that minimizes (\ref{R1d.eq}),  one seeks a minimum-weight path from $s$ to $t$ in
$\Gb$ that has at most $k+1$ edges.  The optimal leaders  are the nodes along this path between $s$ and $t$.  To find this minimum-weight path,
we use a slightly modified implementation of the Bellman-Ford algorithm~\cite{CLRS10}.
The Bellman-Ford algorithm is an iterative algorithm that finds the minimum-weight paths (of any length) from a source node to every other node in the graph.
While there are  more efficient algorithms that solve this same problem, Bellman-Ford offers the benefit that, in each iteration $m$, the algorithm finds the minimum-weight paths of $m$ edges.
Therefore, we can execute the Bellman-Ford algorithm for $k+1$ iterations to find the minimum-weight path of at most $k+1$ edges.
We have made slight modifications to this algorithm to make it possible to retrieve not only the weight of the path but the list of nodes traversed in this path.
Our modified version of Bellman-Ford is detailed in the appendix.

The pseudocode for our $k$-leader selection algorithm is given in Algorithm~\ref{optpath.alg}.  Our algorithm returns the optimal set of leaders of size at most $k$.
A leader set may have cardinality $h < k$ if the inclusion of more than $h$ leaders does not decrease $R(S)$.

\begin{algorithm}
\caption{Algorithm for finding the optimal solution to the $k$-leader selection problem in a ring graph.}  \label{optring.alg}
\begin{algorithmic}
\State \textbf{Input:} $G = (V,E)$, edge weights $\frac{1}{\nu_{ij}}$, $k$
\State \textbf{Output:} Set of leader nodes $S$, error $R(S)$
~\\
\State $L \gets $ weighted Laplacian for $G$

\State $minWeight \gets \infty$
\State $minP \gets \bot$
\For{$i=1 \ldots n$}
\State  $\Gbb \gets$ digraph constructed from $L$ for first leader $i$
\State $(P,w) \gets \textsc{ModifiedBellmanFord}(\Gbb, s_i, t_i, k-1)$
\If{$w < minWeight$}
	\State $minWeight \gets w$
	\State $minP \gets P$
\EndIf
\EndFor
~\\
\State \textbf{Construct leader set from min. weight path:}
\State $S \gets \emptyset$
\For{$v \in minP$}
\If{$v \neq s$ \text{and} $v \neq t$}
\State $S \gets S \cup \{v\}$
\EndIf
\EndFor
\State $R(S) \gets minWeight$
\State \textbf{return} $(S, R(S))$
\end{algorithmic}
\end{algorithm}

\subsection{Leader Selection for a Ring Graph}

In a ring graph with $k$ leaders, the leader-follower system can be decomposed into $k$ independent subsystems (with some possibly consisting of zero nodes).
Each of these subsystems corresponds to a segment of the graph where two leader nodes form the boundaries of this segment. As before, we can define a weighted digraph
where the weight of edge $(u,v)$ is the total steady-state variance of the nodes on the segment between $u$ and $v$ when both $u$ and $v$ are leaders.
The total steady state-variance of the leader-follower system is then given by the weight of a path through this digraph. We now present the details of our $k$-leader selection algorithm for ring graphs. Pseudocode is given in Algorithm~\ref{optring.alg}.

To find the optimal leader set of size at most $k$, we first select one node $i$ as a leader.
We then translate the problem finding the remaining $k-1$ leaders into a problem of finding a minimum weight path of at most $k$ edges
over a weighted digraph.    The digraph construction is described below.
To ensure that our algorithm finds the optimal leader set, the algorithm performs this translation and path-finding for each possible initial leader $i$, $i= 1, \ldots, n$.
The optimal leader set is the set with the minimum weight path among these $n$ minimum weight paths (one for each initial leader selection $i$).

For a given initial leader $i$, we construct its weighted digraph
$\Gbb = (\Vbb, \Ebb, \Wbb)$ as follows.  The vertex set of $\Vbb$ contains a source node $s_i$, a target node $t_i$,
and the vertices in $V$ excepting $i$, i.e.,
\[
\Vbb = \{s_i, t_i\} \cup ( V \setminus \{i\}).
\]
The edge set $\Ebb$  contains directed edges from $s_i$ to every node $v \in ( V \setminus \{i\})$.
The weight of edge $(s_i,v)$ is the total steady-state variance for the nodes between $i$ and $v$ in the clockwise direction on the ring graph when both $i$ and $v$ are leaders and there are no other leader nodes after $i$ and  before $v$ in the clockwise direction.
$\Ebb$ also contains directed edges from every node $v \in (V \setminus \{i\})$ to $t_i$.
The weight of edge $(v,t_i)$ is the total steady-state variance of the nodes between $v$ and $i$ in the clockwise direction on the ring graph when both $i$ and $v$ are leaders and there are no other leader nodes after $v$ and before $i$ in the clockwise direction.
The weight of edge $(u,v)$ is the total variance of the nodes between $u$ and $v$ in the clockwise direction on the ring graph when both $u$ and $v$ are leaders
and there are no other leader nodes following $u$ and preceding $v$.

To compute the weight for edge $(u,v)$, where $u$ precedes $v$ on the ring in the clockwise direction,
 we first construct the matrix $M$ from $L$ by shifting the rows and columns of $L$ so that node
$u$ corresponds to the first row and column of $M$.
The index (row and column of $M$) corresponding to a node $u$ after this shift is $1$,
and the index corresponding to node $v$ is $v - u + 1$ modulo $n$.
The weight of edge $(u,v)$  is given by
\begin{equation} \label{wuv.eq}
w_{uv} = \tr{ {M_{(1, v - u+ 1 \Mod{n})}}^{-1}},
\end{equation}
where $M_{(j,k)}$ is is the sub-matrix of $M$ consisting of the rows and columns between $j+1$ and $k-1$, inclusive.
The weights of edges $(s_i,v)$ , $v \in (V \setminus \{i\})$ can be computed by replacing $u$ with $i$ in the above procedure.
The weights of edges $(u, t_i)$, $u \in (V \setminus \{i\})$ can be computing by replacing $v$ with $i$ in the above procedure.
Note that each of the sub-matrices used in this procedure is a tri-diagonal matrix.  Also, as with the path graph, the weight of an edge $(u,v)$ is 0
if nodes $u$ and $v$ are adjacent in $G$.

Once the graph $\Gbb$ is constructed, the modified Bellman-Ford algorithm is used to find the minimum-weight path from $s_i$ to $t_i$ of at most $k$ edges.
The weight of this path is the minimal $R(S)$ when $i \in S$ and $|S| \leq k$.  The leader set $S$ for this error consists of node $i$ and the nodes along this
path, excepting $s_i$ and $t_i$.  By finding the optimal leader that contains $i$, for each $i \in V$,
the algorithm is able to identify the initial leader $i$ that minimizes $R(S)$ and thus the optimal solution to the $k$-leader selection problem.

\subsection{Algorithm Analysis}
We now analyze the computational complexity of our leader selection algorithms.

\begin{theorem}
For a path graph $G$ with $n$ nodes, the $k$-leader selection algorithm identifies the leader set $S$, with $|S| \leq k$, that minimizes $R(S)$
in $O(n^3)$ time.
\end{theorem}
\begin{proof}
For the path graph, our algorithm consists of two phases, each of which is performed once.  The first phase is the construction of the digraph $\Gb = (\Vb, \Eb)$.
The edge set $\Eb$ consists of $n$ edges with $s$ as their source (one to each $v \in V$), $n$ edges with $t$ as their sink (one from each $v \in V$),
and one edge from each $u \in V$ to each $v \in V$ with $u < v$.  Thus $| \Eb | \in O(n^2)$.
To find each edge weight, we must find the diagonal entries of the inverse of a tridiagonal matrix of size at most $(n-1) \times (n-1)$.
These diagonal entries can be found in $O(n)$ operations~\cite{RH91}.  Therefore the digraph $\Gb$ can be constructed in $O(n^3)$ operations.

The second phase of the algorithm is to find the shortest path of at most $k+1$ edges from $s$ to $t$ in $\Gb$.  For a graph with $m$ edges,
the Bellman-Ford algorithm finds the minimum-weight-path of length at most $h$  edges in $O(hm)$ operations~\cite{CLRS10}.
Therefore the second phase of our algorithm has complexity $O(kn^2)$.

Combining the two phases of the algorithm we arrive a computational complexity of $O(n^3)$.
\end{proof}
$ $\\
\begin{theorem}
For a ring graph $G$ with $n$ nodes, the  $k$-leader selection algorithm identifies the leader set $S$, with $|S| \leq k$, that minimizes $R(S)$
in $O(kn^3)$ time.
\end{theorem}
\begin{proof}
In the ring algorithm, $n$ weighted digraphs $\Gbb$, $i=1,\ldots, n$ are constructed, and a shortest-path algorithm is executed on each digraph.
To construct these digraphs, first, the weight of each edge $(u,v)$, $u,v \in V, u \neq v$ is computed according to
according to (\ref{wuv.eq}).  For each computation, the shift operation can be performed in $O(n)$ to obtained a tridiagonal matrix.
The trace of the inverse of the sub-matrix can also be found in $O(n)$. Thus, the weight of all pairs $(u,v)$ can be computed in
$O(n^3)$.  These edge weights are used in every digraph $\Gbb$ and can be looked up in constant time (for example, by storing them in an $n \times n$ matrix).

To construct a digraph where edge weights can be computed in constant time requires $O(|V| + |E|)$ operations.  Therefore, each $\Gbb$ can be constructed in $O(n^2)$.
There are $n$ such digraphs, so the construction of all digraphs requires $O(n^3)$ time.

Finally, for each digraph, the Bellman-Ford algorithm finds the minimum-weight path of at most $k$ edges in $O(kn^2)$ time.  Thus, to find the minimum-weight path over these $n$ minimum-weight paths requires $O(kn^3)$ time.

Combining all steps of the algorithm, we obtain a running time of $O(kn^3)$.

\end{proof}

\section{Computational Examples} \label{examples.sec}
In this section, we explore the results of our $k$-leader selection algorithms on several example graphs.
For comparison, we also show results from the greedy leader selection algorithm presented in~\cite{CBP14}.

The greedy algorithm consists of at most $k$ iterations. In each iteration, a node $s$  is selected
such that $s = \arg \min_{s \in V \setminus S} \Merr{S \cup \{s\}}$.  If $\Merr{S \cup \{s\} }= \Merr{S}$,
then the algorithm returns $S$ (with less than $k$ leaders).  Otherwise, $s$ is added to $S$ to complete the iteration.
If the algorithm completes all $k$ iterations, then it returns the leader set $S$ containing $k$ elements.
 Pseudocode for the greedy leader selection algorithm is given in Algorithm~\ref{greedy.alg}.

\begin{algorithm}
\caption{Greedy leader selection algorithm.} \label{greedy.alg}
\begin{algorithmic}
\State \textbf{Input:} $G = (V,E)$, edge weights $\frac{1}{\nu_{ij}}$, $k$
\State \textbf{Output:} Set of leader nodes $S$
~\\
\State $S \gets \emptyset$, $i \gets 1$
\While{$i \leq k$}
\State $s \gets \arg \min_{v \in V \setminus S} \Merr{S \cup \{v\}}$
\If{$\Merr{S^*} - \Merr{S \cup \{s\}} \leq 0$}
\State \textbf{return $S$}
\Else
\State $S \gets S \cup \{s\}$
\EndIf
\State $i \gets i + 1$
\EndWhile
\State \textbf{return} $S$
\end{algorithmic}
\end{algorithm}

The greedy leader selection algorithm generates a  leader set $S$ of size at most $k$ such that,
\[
\Merr{S} \leq \left ( 1 - \left(\frac{k - 1}{k}\right)^k\right)\Mopt +  \frac{1}{e}\Mmax,
\]
where $\Mmax \eqdef \max_{i \in V} \Merr{i}$~\cite{CBP14}.
As far as we are aware, the greedy algorithm is the only previously proposed algorithm that gives provable guarantees
on the optimality of the leader set.

We have implemented both algorithms in Matlab.  For each graph, we select edge weights $\nu_{ij}$ uniformly at random
from the interval $(0,1)$.

\begin{figure}
\centering
        \begin{subfigure}{.9\linewidth}
            \includegraphics[scale=.42]{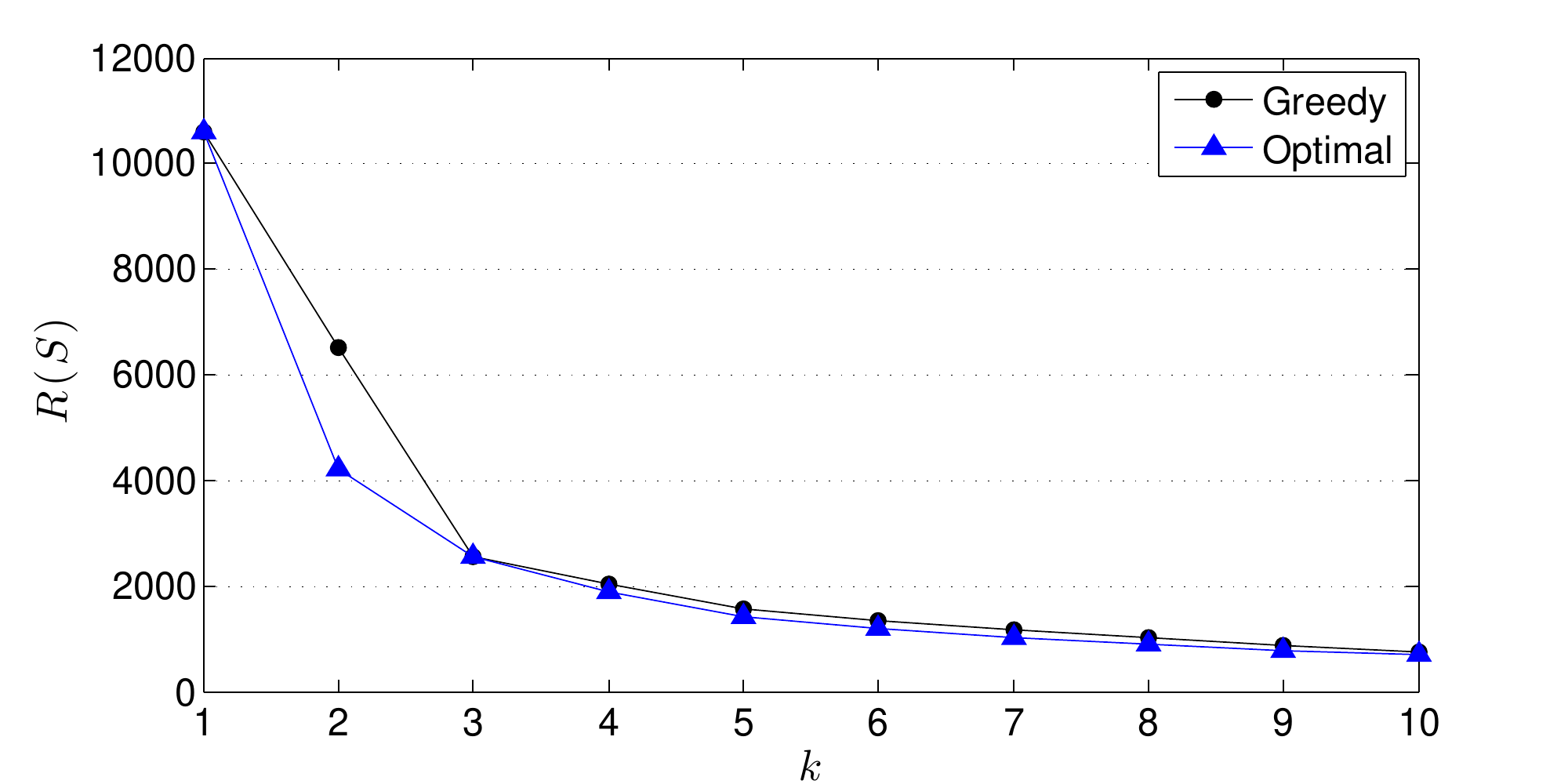}
                \caption{Total steady-state variance for leader sets of sizes $k=1,\ldots, 10$.}
                \label{lineabs.fig}
        \end{subfigure}%
        \\
        ~ 
        \begin{subfigure}{.9\linewidth}
              \includegraphics[scale=.42]{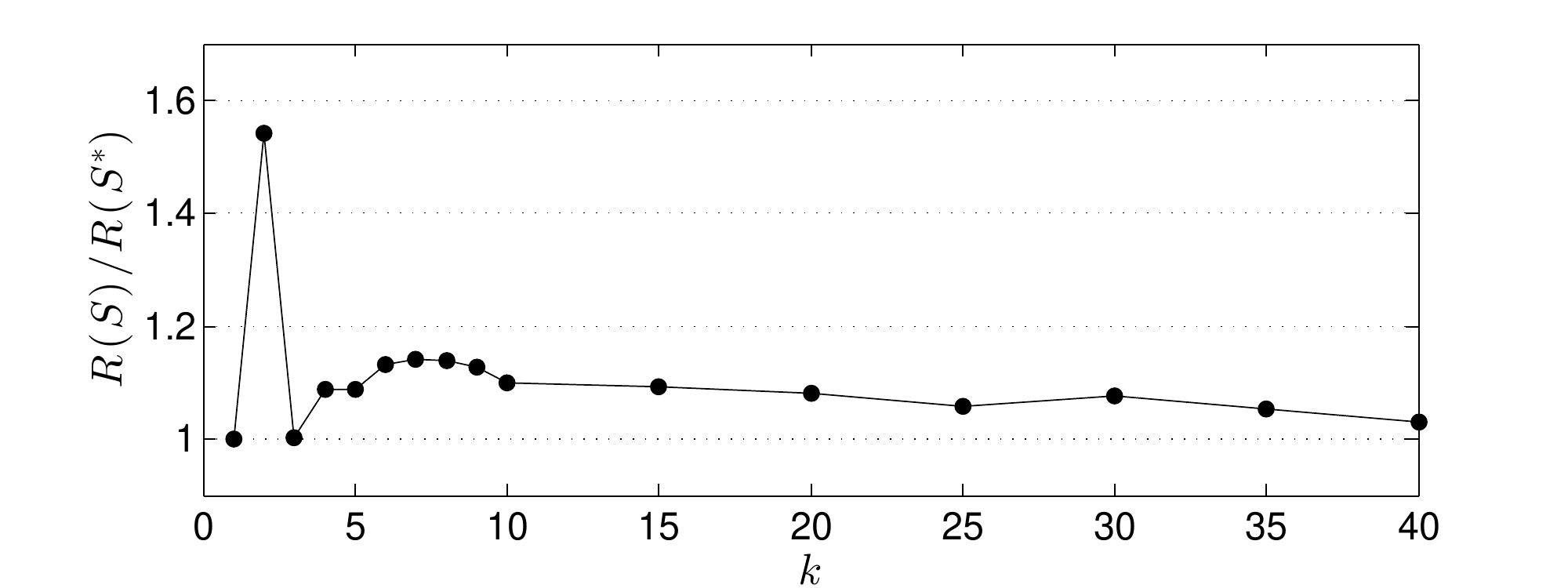}
                \caption{Total variance of leader set $S$ of size $k$ produced by greedy algorithm, relative to total variance for optimal leader set $S^*$ of size $k$.}\label{linerel.fig}
        \end{subfigure}
        \caption{A comparison of the our optimal leader selection algorithm with the greedy leader selection algorithm on a 400 node path graph.} \label{line.fig}
\end{figure}

The results for a 400 node path graph are shown Figure~\ref{line.fig}.
Figure~\ref{lineabs.fig} gives the total variance $R(S)$ for the greedy algorithm and our optimal algorithm for various leader set sizes $k$.
Figure~\ref{linerel.fig} shows total variance $R(S)$ of the leader set selected by the greedy algorithm relative to $R(S^*)$, where $S^*$ is the optimal leader set found by our algorithm.  Both algorithms find the optimal leader for $k=1$.   An interesting observation is that for $k=2$, the greedy algorithm demonstrates its worst relative performance, while for $k=3$, the greedy algorithm finds a leader set that is nearly optimal.  We believe that this is due to a kind of symmetry in path graphs.
For $k=1$, the optimal leader $\ell_1$ is the weighted median of the path graph (see \cite{P14}), and both algorithms select this leader.  For $k=2$, the optimal leaders are
nodes $\ell_2$ and $\ell_3$ near opposite  ends of the graph, (though not at the ends).  The greedy algorithm selects leader $\ell_1$, which is a poor choice for $k=2$, and one additional leader near the boundary of one side of the path graph, resulting in a larger total variance For $k=3$, the greedy algorithm is able to compensate for the poor choice of leader 1 by choosing a third leader near the other boundary of the graph.
We note that, overall the greedy algorithm yields leader sets whose performance is fairly close to optimal.  As the number of leaders increases,
the total variance increases for both leader selection algorithms.  The relative error of the greedy algorithm does not appear to vanish as $k$ increases.

\begin{figure}
\centering
        \begin{subfigure}{.9\linewidth}
            \includegraphics[scale=.42]{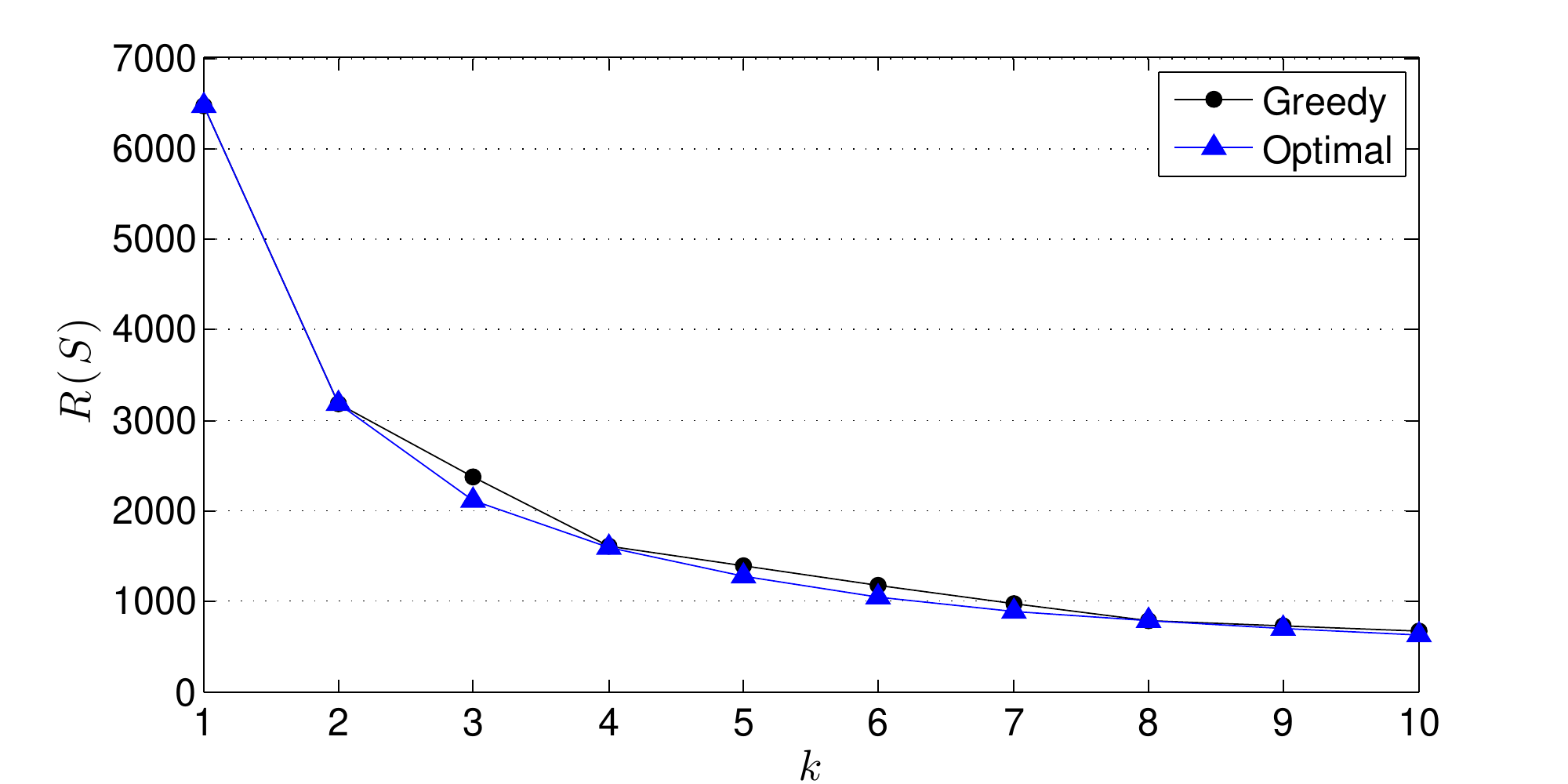}
                \caption{Total steady-state variance for leader sets of sizes $k=1,\ldots, 10$.}
                \label{ringabs.fig}
        \end{subfigure}%
        \\
        ~ 
        \begin{subfigure}{.9\linewidth}
              \includegraphics[scale=.42]{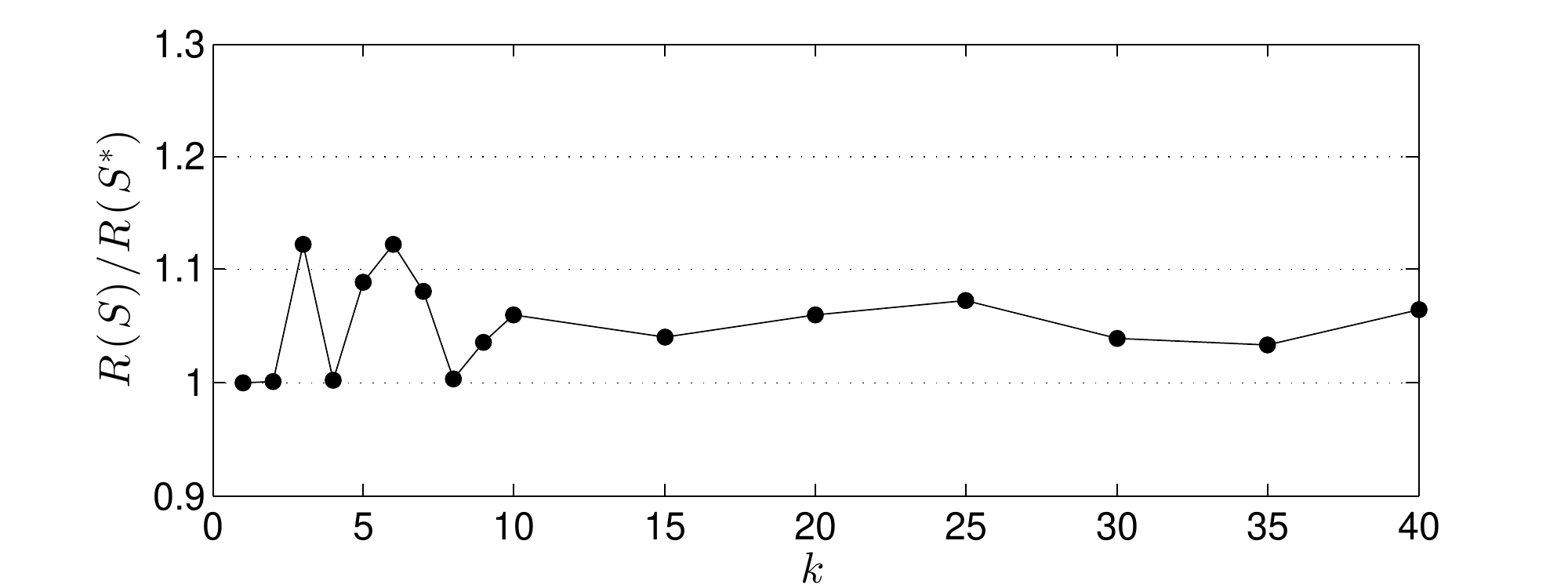}
                \caption{Total variance of leader set $S$ of size $k$ produced by greedy algorithm, relative to total variance for optimal leader set $S^*$ of size $k$.}\label{ringrel.fig}
        \end{subfigure}
        \caption{A comparison of the our optimal leader selection algorithm with the greedy leader selection algorithm on a 400 node ring graph.}\label{ring.fig}
\end{figure}

The results for a 400 node ring graph are shown in Figure~\ref{ring.fig}.
As before, Figure~\ref{ringabs.fig} gives the total variance $R(S)$ for the greedy algorithm and our optimal algorithm for various leader set sizes $k$,
and Figure~\ref{ringrel.fig} shows total variance $R(S)$ of the leader set selected by the greedy algorithm relative to $R(S^*)$, where $S^*$ is the optimal leader set as identified by our algorithm.
For  all $k$ greater than 1, the greedy algorithm selects a sub-optimal leader set.  The most interesting results are for smaller values of $k$.
The greedy algorithm exhibits worse performance when $k$ is odd than when $k$ is even.  We believe this is due to a similar phenomenon as observed when $k=2$ and $k=3$ in the path graph.
The greedy algorithm compensates for previous poor choices when an even numbered leader selected.
For larger values of $k$,
the performance of the greedy algorithm appears to stabilize at around 1.05 times the optimal $R(S^*)$.
While the greedy algorithm performs slightly worse for the ring graph than it does for the path graph, overall, the greedy solution yields good approximations for larger values of $k$.

\section{Conclusion} \label{conclusion.sec}
We have investigated the problem of optimal $k$-leader selection in noisy leader-follower consensus systems.
A naive solution to the leader selection problem has combinatorial complexity; however, it is unknown whether the leader selection problem
is NP-Hard, in general, or if efficient polynomial-time solutions can be found.
In this work, we have taken a step towards addressing this open question.
We have shown that, in one-dimensional weighted graphs, namely path graphs and ring graphs, the $k$-leader selection problem can be solved in polynomial time (in both $k$ and the network size $n$).
Further, we have given an $O(n^3)$ solution for optimal $k$-leader selection in path graphs and an $O(kn^3)$ solution for optimal $k$-leader selection in ring graphs.

It is our intuition that more efficient algorithms for leader-selection in one-dimensional networks can be found.  We plan to explore this in future work.
Further, we plan to extend our analysis and algorithm development to the problem of efficient $k$-leader selection in other network topologies. Finally, we will explore
using similar algorithmic techniques for leader selection in other dynamics, including networks where leaders are also subject to stochastic noise and networks in which there are no stochastic disturbances and the goal is to select leaders that maximize the convergence speed.

\section*{APPENDIX}
Pseudocode for our modified Bellman-Ford algorithm is given in Algorithm~\ref{mbf.alg}.
The input to the algorithm is a weighted digraph ${\cal G}=({\cal V}, {\cal E}, {\cal W})$, source and target nodes $s$ and $t$, and the maximum number of hops $H$ allowed in the path to be found.
The algorithm performs $H$ iterations; in each iteration $m$, it finds the minimum-weight path of exactly $m$ hops from node $s$ to every other node.
If no such path exists, the path weight is infinite. 
\begin{algorithm}[h]
\caption{Modified Bellman-Ford algorithm.}\label{mbf.alg}
\begin{algorithmic}
\State \textbf{Input:} weighted digraph ${\cal G} = ({\cal V},{\cal E},{\cal W})$, source node $s$,  target node $t$,
number of edges $H$
\State \textbf{Output:} List $P$ of vertices on min. weight path from $s$ to $t$ of at most $H$ edges, weight $w$ of path $P$
~\\
\State \emph{ /* $d_{i}(v)$ stores the weight of the minimum-weight path from node $s$ to node $v$ that contains exactly $m$ edges. */}
\State $d_{0}(s) \gets 0$
\State $d_{m}(s) \gets \infty$ for $m = 1\ldots k$
\State $d_{m}(v) \gets \infty$ for $v \in V, v \neq s, m = 0 \ldots k$
~\\
\State \emph{ /* $\pi_m(v)$ stores the ID of the predecessor of node $v$ in the minimum-weight path of $m$ edges. */ }
\State $\pi_{m}(v) \gets \perp$ for $v \in V, m = 0 . . . k$
~\\
\State \emph{/* Path $P$ initialized to empty list */}
\State $P \gets []$
~\\
\State \textbf{Find minimum-weight paths of all lengths up to $H$ from $s$ to all other nodes:}
\For{$m = 1 \ldots H$}
\For{$v \in V$}
\State $d_{m}(v) = \displaystyle \min_{u \in N(v)}\{d_{m-1}(u) + w(u,v) \}$
\State $\pi_{m}(v) =  \displaystyle  \arg\min_{u \in N(v)}\{d_{m-1}(u) + w(u,v) \}$
\EndFor
\EndFor
~\\
\State \textbf{Identify min. weight path of at most $H$ edges:}
\State $ L = H$
\While{$d_{L}(t) = d_{L-1}(t)$}
	\State $L \gets L-1$
\EndWhile
\State $P$.append($t$);
\State $u \gets \pi_L(t)$
\State $P$.append($u$)
\For{$i=(L-1) \ldots 2$}
	\State $u \gets \pi_{i}(u)$
	\State $P$.append($u$)
\EndFor
\State \textbf{return} $(P,d_L(t))$
\end{algorithmic}
\end{algorithm}

To find the path of length $m$ from  $s$ to a node $v$, the algorithm examines each node $u \in N(v)$ and finds the minimum over the weight of path from $s$ to $u$ of length $m-1$ plus the weight of edge $(u,v)$.   Each iteration of the algorithm requires a number of operations on the order of the number of edges in the graph.  After $H$ iterations, the algorithm finds the minimum-weight path from $s$ to $t$ of at most $H$ edges over the computed paths from $s$ to $t$ of $m$ edges, $m = 1\dots H$.
The running time of this algorithm is $O(H |{\cal E}|)$.  We refer the reader to~\cite{CLRS10} for more details.

%

\vspace{-.3cm}

\bibliographystyle{IEEETran}
\bibliography{kleader_ecc_submitted}

\end{document}